\documentclass{amsart}

\usepackage{graphicx,mathrsfs,amssymb}

\usepackage[pdfdisplaydoctitle=true,colorlinks=true,urlcolor=blue,citecolor=blue,linkcolor=blue,pdfstartview=FitH,pdfpagemode=None,bookmarksnumbered=true]{hyperref}
\usepackage{cmap} 

\newtheorem{theorem}{Theorem}[section]
\newtheorem{lemma}[theorem]{Lemma}

\newtheorem{proposition}[theorem]{Proposition}
\newtheorem{remark}[theorem]{Remark}
\newtheorem{definition}[theorem]{Definition}

\theoremstyle{remark}
\newtheorem{example}[theorem]{Example}

\numberwithin{equation}{section}

\newcommand{\cz}{{\mathbb C}}

\newcommand{\nz}{{\mathbb N}}

\newcommand{\rz}{{\mathbb R}}

\newcommand{\calD}{\mathcal{D}}
\newcommand{\calE}{\mathcal{E}}

\newcommand{\calK}{\mathcal{K}}

\newcommand{\calM}{\mathcal{M}}

\newcommand{\calV}{\mathcal{V}}
\newcommand{\calW}{\mathcal{W}}

\newcommand{\fraka}{\mathfrak{a}}
\newcommand{\frakE}{\mathfrak{E}}
\newcommand{\frakg}{\mathfrak{g}}

\newcommand{\scrC}{\mathscr{C}}
\newcommand{\scrD}{\mathscr{D}}
\newcommand{\scrF}{\mathscr{F}}

\newcommand{\scrL}{\mathscr{L}}
\newcommand{\scrS}{\mathscr{S}}

\newcommand{\cicomp}{\mathcal{C}^\infty_{\text{\rm comp}}}

\newcommand{\dbar}{d\hspace*{-0.08em}\bar{}\hspace*{0.1em}}
\newcommand{\eps}{\varepsilon}

\newcommand{\forget}[1]{}

\newcommand{\lra}{\longrightarrow}

\newcommand{\re}{\mathrm{Re}\,}

\newcommand{\rpbar}{\overline{\rz}_+}
\newcommand{\skp}[2]{\langle#1,#2\rangle}

\newcommand{\smsum}{\mathop{\mbox{\large$\sum$}}}
\newcommand{\spk}[1]{\langle#1\rangle}
\newcommand{\st}{\mbox{\boldmath$\;|\;$\unboldmath}}

\newcommand{\wh}{\widehat}
\newcommand{\wt}{\widetilde}

\begin{document}
\title[NATURAL DOMAINS FOR EDGE-DEGENERATE DIFFERENTIAL OPERATORS]{NATURAL DOMAINS FOR 
EDGE-DEGENERATE\\[2mm] DIFFERENTIAL OPERATORS}

\author{J\"org Seiler}
\address{Loughborough University, School of Mathematics, Leicestershire LE11 3TU, 
United Kingdom}
\email{j.seiler@lboro.ac.uk}

\begin{abstract}
We study cone differential operators on the half-axis and edge-degenerate differential operators on a half-space. 
We construct subspaces of edge Sobolev spaces that can be considered as natural domains for edge-degenerate 
operators and indicate how they can be used in the study of boundary problems for edge-degenerate operators. 
\end{abstract}

\maketitle
\section{Introduction}\label{sec:intro}

Let $\Omega$ be compact domain in Euclidean space with smooth boundary $Y:=\partial\Omega$. 
An elliptic differential operator $A$ of order $\mu\in\nz$ on $\Omega$ induces mappings  
 $$A:H^{s}(\Omega)\lra H^{s-\mu}(\Omega), \qquad s\in\rz.$$
Any of these maps generally fails to be a Fredholm operator, and for this reason one seeks 
to complete $A$ with $($differential$)$ boundary conditions $T$ to a map 
 $$\begin{pmatrix}A\\T\end{pmatrix}:H^{s}(\Omega)\lra \begin{matrix}H^{s-\mu}(\Omega)\\ 
    \oplus\\ H^{s-\mu}(\partial\Omega,\cz^k)\end{matrix},
    \qquad s>\mu-\frac{1}{2}$$
(for notational convenience we have unified orders on the right-hand side by 
applying suitable order reductions on the boundary; the requirement on $s$ arises from the 
fact that the map of restricting smooth functions from the domain to the boundary extends  
continuously to a map from $H^s(\Omega)$ to $H^{s-1/2}(Y)$ only for $s>1/2$). 
A pseudodifferential calculus containing such kind of operators and parametrices of  
elliptic problems is Boutet de Monvel's algebra \cite{Bout}. Whether one can find boundary 
conditions completing $A$ to a `Shapiro-Lopatinskij' elliptic boundary value problem depends 
on the boundary symbol of $A$, 
 $$\sigma_\partial^\mu(A)(y,\eta):H^s(\rz_+)\lra H^{s-\mu}(\rz_+)$$
which is a family of Fredholm operators defined on the co-sphere bundle $S^*Y$ of the boundary. 
It induces an `index element' in the $K$-group $K(S^*Y)$. There exist elliptic boundary conditions precisely 
when this index element satisfies the Atiyah-Bott condition, i.e., belongs to 
$\pi^*K(Y)$, the pull-back of the $K$-group over the boundary under the natural 
projection $\pi:S^*Y\to Y$. If $A$ in local coordinates $(y,t)\in\rz^q\times\rz_+$ near the boundary 
has the form $\sum_{j+|\alpha|\le\mu}a_{j\alpha}(y,t)D_y^\alpha D_t^j$, 
the boundary symbol is given by 
 $$\sigma_\partial^\mu(A)(y,\eta)=\sum_{j+|\alpha|=\mu}a_{j\alpha}(y,0)\eta^\alpha D_t^j.$$

The question arises, if and how one could organize a corresponding calculus when the 
differential operators are not smooth up to the boundary but have a more singular behaviour. 
The structure we have in mind here are `edge-degenerate' differential operators, i.e., $A$ 
is away from the boundary a usual differential operator, but near the boundary is of the 
form (in local coordinates)
 $$A=t^{-\mu}\sum_{j+|\alpha|=0}^\mu a_{j\alpha}(y,t)(tD_y)^\alpha(-t\partial_t)^j,$$
with coefficients $a_{j\alpha}$ which are smooth up to the boundary. Note that any usual 
differential operator can be rewritten in this degenerate form, but not vice versa. This 
particular degeneracy arises naturally in the analysis of differential operators on manifolds 
with edges where the natural `geometric' operators like the Laplacian are of this form 
(though in this case a neighborhood of the edge is a cone bundle with fibre $\rz_+\times X$ 
for a closed manifold $X$, rather than $X=\{\mathrm{point}\}$ as in the case of a bounded 
domain). The usual Sobolev spaces are not the natural spaces to be used in 
this setting and it is not clear which kind of `boundary conditions' one should pose -- 
if possible at all. 

In the 1980's Schulze developed a pseudodifferential calculus -- the `edge algebra' -- 
adapted to edge-degenerate operators (not only on bounded domains but on manifolds with edges), 
see for example the monographs \cite{ReSc} and \cite{Schu1}. The 
corresponding scale of edge Sobolev spaces $\calW^{s,\gamma}(\Omega)$, $s,\gamma\in\rz$, 
(for a precise definition see Section \ref{sec:4.1}) is different from the standard one. 
There are many similarities, but also essential differences, between this calculus and that of 
Boutet de Monvel. The role of the boundary symbol for example is now played by the principal edge 
symbol, defined as 
  $$\sigma^\mu_\wedge(A)(y,\eta)
    =t^{-\mu}\sum_{j+|\alpha|=0}^\mu a_{j\alpha}(y,0)(t\eta)^\alpha(-t\partial_t)^j:
    \calK^{s,\gamma}(\rz_+)\lra \calK^{s-\mu,\gamma-\mu}(\rz_+).$$
Here, $\calK^{s,\gamma}(\rz_+)$ refers to certain Sobolev spaces on the half-axis, cf. 
Definition \ref{def:conesob}. A main difference between the two calculi concerns the type of boundary 
respectively edge conditions. In the smooth setting restriction to the boundary is a well-defined 
operation for the standard Sobolev spaces but this is not anymore the case for the edge spaces. 
Correspondingly the boundary conditions are of different nature. In \cite{KSS} Kapanadze, 
Schulze and the author constructed an extended edge algebra using an enlarged scale 
of edge Sobolev spaces that allowed to generalize the restriction-to-the-boundary mappings, 
and so to interpret Boutet de Monvel's algebra as a subalgebra in this larger 
calculus. The main idea is to replace Taylor asymptotics of functions at the boundary by a more 
general type of asymptotic behaviour. While the standard boundary conditions can be 
interpreted, roughly speaking, as functionals acting on the Taylor coefficients, the 
generalized boundary conditions do act on the coefficients of the more general expansions. 

Though this calculus extends the one of Boutet de Monvel, it appears being somewhat too 
coarse to tackle in full generality the above posed question -- what are natural domains 
and how to find associated Fredholm problems. In a certain sense the enlarged spaces are too big. 
In this paper we discuss how to further refine the calculus from \cite{KSS} to achieve this 
goal. We shall present the basic idea in a model situation where the operators 
are defined in a half space rather than on a bounded domain, and have constant coefficients along 
the boundary (actually, it is enough that the first $\mu$ conormal symbols of the operator are 
$y$-independent). 

In Section \ref{sec:3} we discuss closed extensions of elliptic cone differential 
operators on the half-line (actually, all the constructions extend to also cover the case 
of an infinite cone $\rz_+\times X$ over a non-trivial cone base $X$ rather than the half-axis). 
The analysis of such extensions was initiated by Lesch \cite{Lesc} and later on refined and 
extended by other authors, see for example Gil, Mendoza \cite{GiMe}, Schrohe, Seiler \cite{SchrSe}, 
Gil, Krainer, Mendoza \cite{GKM2} and Coriasco, Schrohe, Seiler \cite{CSS2}. We reprove here some 
of the known results, using a formalism following \cite{SchrSe}. 
In Section \ref{sec:4} we use this approach to construct natural domains for 
edge-degenerate differential operators. We show that they naturally arise 
as subspaces of edge Sobolev spaces with asymptotics (the latter known from the standard edge 
calculus) and we construct natural pseudodifferential projections onto these spaces. The  
principal symbols of these projections yield pointwise projections onto the maximal domain of the 
principal edge symbol, viewed as a family of cone differential operators. 
In the last part of Section \ref{sec:4} we indicate how these projections can be used 
to formulate a refined version of the calculus from \cite{KSS} in `projected subspaces'.  

\section{The Laplacian on a Half-space}\label{sec:2}

The meaning of this section is to discuss a simple example, the Laplacian $\Delta$ on the 
half space $\Omega:=\rz^q\times\rz_+$, in a way that motivates our approach later on. 
So, we will give some fancy looking explanation why it is natural to choose $H^2(\Omega)$ 
as the domain for $\Delta$ in $L^2(\Omega)$. 

Let us introduce the following family of maps $\kappa_\lambda$, $\lambda>0$, that 
acts on functions $($or distributions$)$ on $\Omega$ by 
 $$(\kappa_\lambda u)(y,t)=\lambda^{1/2}u(y,\lambda t), \qquad (y,t)\in\rz^q\times\rz_+.$$
Thus $\kappa_\lambda$ acts essentially as a dilation on the $t$-variable; the factor 
$\lambda^{1/2}$ makes $\kappa_\lambda$ an isometrie on $L^2(\Omega)$. Obviously, 
the $\kappa_\lambda$ form a group, i.e., $\kappa_\lambda\kappa_\rho=\kappa_{\lambda\rho}$ 
and $\kappa_1=\mathrm{id}$. We can now define the operator 
 $$L:=\scrF^{-1}_{\eta\to y}\, \kappa_{\spk{\eta}}^{-1}\, \scrF_{y^\prime\to\eta}$$
where we write $\spk{\eta}=(1+|\eta|^2)^{1/2}$ and $\scrF$ is the standard Fourier transform. 
By a direct (formal) calculation it is then easy to see that conjugating $1-\Delta$ with $L$ gives 
 $$\wt{A}:=L\,(1-\Delta)\,L^{-1}=(1-\Delta_y)(1-\partial_t^2).$$ 
Thus we have split $1-\Delta$ in two operators, one along the boundary and one in direction normal 
to the boundary. Now the maximal domain of $(1-\Delta_y)$ in $L^2(\rz^q)$ is just 
$H^2(\rz^q)$, while the maximal domain of $(1-\partial_t^2)$ in $L^2(\rz_+)$ can be shown to be 
$H^2(\rz_+)$. So it is natural to take $H^2(\rz^q,H^2(\rz_+))$ as domain for $\wt{A}$ in 
$L^2(\Omega)=L^2(\rz^q,L^2(\rz_+))$. So the natural domain for $\Delta$ itself is 
$L^{-1}H^2(\rz^q,H^2(\rz_+))$ which can be shown to coincide with $H^2(\Omega)$, 
cf.\ Section \ref{sec:4.1}. 

This approach can be used to find natural domains for general edge-degenerate operators. 
Conjugation with $L$ amounts to a splitting of operators where on $\rz_+$ we will obtain 
Fuchs-type differential operators. We study the maximal domains of such operators in the next section.

\section{Fuchs-type Differential Operators on the Half-axis}\label{sec:3}

In this section we let $A$ denote an elliptic Fuchs-type differential operator on the 
half-axis. More precisely, we assume that $A$ is a differential operator of order $\mu$ 
with smooth coefficients, that near $t=0$ has the form 
\begin{equation}\label{eq:conediff}
 A= t^{-\mu}\sum_{j=0}^\mu a_j(t)(-t \partial_t)^j,\qquad a_j\in\scrC^\infty(\rpbar)
\end{equation}
(in case of a non trivial cone base $X$ the coefficient functions $a_j(t)$ take values in the differential 
operators on $X$ of order at most $\mu-j$). 
We can write $A=a(t,D_t)$ with a symbol $a(t,\tau)$ which is a polynomial in $\tau$. 
We shall assume that 
 $$|\partial_t^j\partial_\tau^k a(t,\tau)|\le C_{jk}\spk{\tau}^{\mu-k}$$
uniformly in $t\ge 1$ and $\tau\in\rz$ for any integers $j$ and $k$. We also assume that this operator 
is elliptic in the following sense:  
\forget{ 
$A$ is elliptic at infinity in the sense that for suitable constants $C$ and $R$ 
 $$|a(t,\tau)^{-1}|\le C\spk{\tau}^{-\mu}\qquad \forall\;(t,\tau)\ge R$$
whenever $t\ge 1$, the principal symbol $\sigma^\mu_\psi(a)(t,\tau)$ never vanishes for $\tau\not=0$, 
and the `rescaled'  principal symbol $t^\mu\sigma^\mu_\psi(a)(t,t^{-1}\tau)$ never vanishes 
for $\tau\not=0$ and all $t\ge 0$. 
}
\begin{itemize}
 \item[(i)] There are constants $C$ and $R$ such that for $t\ge1$ 
   $$|a(t,\tau)^{-1}|\le C\spk{\tau}^{-\mu}\qquad \forall\;(t,\tau)\ge R,$$
 \item[(ii)] the principal symbol $\sigma^\mu_\psi(a)(t,\tau)$ never vanishes for 
   $\tau\not=0$,  
 \item[(iii)] the rescaled symbol $t^\mu\sigma^\mu_\psi(a)(t,t^{-1}\tau)$ never vanishes 
   for $\tau\not=0$. 
\end{itemize}

We shall now derive explicit descriptions of the maximal extension of $A$ when considered 
as an unbounded in $L^2(\rz_+)$, initially defined on the space of smooth compactly supported 
test functions (the results extend in a straightforward way to the framework of 
Fuchs-type operators on an infinite cone $\rz_+\times X$ with a closed cross-section $X$ 
of arbitrary dimension). 

\subsection{Cone Sobolev spaces} \label{sec:3.1}

We need to recall the definitions of certain cone Sobolev spaces on $\rz_+$. We fix a cut-off function 
$\omega\in\scrC^\infty_0(\rpbar)$, i.e., $\omega$ is smooth and compactly supported and $\omega\equiv1$ in 
some neighborhood of $t=0$. 

\begin{definition}\label{def:conesob}
For $s$ a non negative integer and $\gamma\in\rz$ let $\calK^{s,\gamma}(\rz_+)$ denote the space of all 
distributions satisfying $(1-\omega)u\in H^s(\rz_+)$ and 
 $$t^{-\gamma}(t\partial_t)^j(\omega u)(t)\in L^2(\rz_+,dt)\qquad \forall\; j\le s.$$
\end{definition}

Note that $\gamma$ indicates a power weight in $t$ for $t\to 0$. 
This spaces can be equipped with the structure of a Hilbert space and the definition can also be extended to cover 
arbitrary real $s\in\rz$. We shall omit any details. Note that $\calK^{0,0}(\rz_+)$ coincides with the space 
$L^2(\rz_+,dt)$, but that $\calK^{s,\gamma}(\rz_+)$, $s\not=0$, is different from $H^s(\rz_+)$ for any chioce of 
$\gamma$. In the particular case $s=\gamma\ge0$ with $s-1/2\not\in\nz_0$ it can be shown that 
$\calK^{s,s}(\rz_+)$ coincides with the closure of $\scrC^\infty_0(\rz_+)$ in $H^s(\rz_+)$. 

The weighted spaces are natural for cone differential operators, in the sense that $A$ from 
\eqref{eq:conediff} induces continuous mappings 
 $$A:\calK^{s,\gamma}(\rz_+)\lra\calK^{s-\mu,\gamma-\mu}(\rz_+),\qquad s,\gamma\in\rz.$$

\begin{definition}\label{def:coneasymp}
Let $\gamma\in\rz$ and $\theta>0$. Then $\mathrm{As}(\gamma,\theta)$ consists of all finite subsets  
$S\subset\cz\times\nz_0$ such that $1/2-\gamma-\theta<\re p<1/2-\gamma$ for any point $(p,n)\in S$ and 
such that to any $p\in\cz$ there is at most one element $(p,n)$ belonging to $S$. We define the function space
$\calE_S\subset\scrC^\infty(\rz_+)$ as  
 $$\calE_S=\Big\{t\mapsto \omega(t)\sum_{i=0}^m\sum_{j=0}^{n_i}a_{ij}t^{-p_i}\log^j t \mid a_{ij}\in\cz\Big\},$$
provided $S=\{(p_0,n_0),\dots,(p_m,n_m\}$. 
\end{definition}

These spaces arise natural in the formulation of elliptic regularity for cone differential operators and below in 
the description of their closed extensions. Note that $\calE_S$ is finite-dimensional, hence carries a natural topology. 
In case $S=\{(-i,0)\st i=0,\ldots,m\}$ the space $\calE_S$ can be interpreted as the space of Taylor polynomials 
of degree $m$, and in this sense the described asymptotic structure is a generalization of Taylor asymptotics. 

\subsection{The maximal domain of a cone differential operator}\label{sec:3.2}

Let $A$ be as in \eqref{eq:conediff} and associate with $A$ its model cone operator $\wh{A}$
which is defined by 
\begin{equation}\label{eq:modelcone}
 \wh{A}= t^{-\mu}\sum_{k=0}^\mu a_k(0)(-t \partial_t)^k
\end{equation}
on $\rz_+$. We shall now describe the spaces 
\begin{align*}
 \scrD_{\max}(\wh{A})&=\{u\in L^2(\rz_+)\st \wh{A}u\in L^2(\rz_+)\}, \\
 \scrD_{\max}(A)&=\{u\in L^2(\rz_+)\st Au\in L^2(\rz_+)\} 
\end{align*}
and a canonical relation between them which is due to \cite{GiMe}, \cite{GKM2}. 
For convenience of notation we work in $L^2(\rz_+)=\calK^{0,0}(\rz_+)$; all what will 
be said has straightforward reformulations in the case $\calK^{s,\gamma}(\rz_+)$ with $s,\gamma\in\rz$. 

We shall need the sequence of so-called conormal symbols of $A$, defined by 
\begin{equation}\label{eq:conormal}
 f_\ell(z)=\sum_{j=0}^\mu a_j^{(\ell)}z^j,\qquad 
 a_j^{(\ell)}:=\frac{1}{\ell!}\frac{d^\ell a_j}{dt^\ell}(0). 
\end{equation}
These are polynomials in the complex variable $z$. Due to the ellipticity of $A$ the principal conormal 
symbol $f_0$ is different from zero, hence $f_0^{-1}$ is a meromorphic function (in case of a non trivial 
cone base, $f_0$ is a holomorphic function with values in the $\mu$-th order differential operators which 
turns out to be meromorphically invertible, any vertical strip in the complex plane of finite width only 
containing finitely many poles; the Laurent coefficients are then smoothing pseudodifferential operators on $X$) . 
In case $f_0^{-1}$ has no pole with real part equal to $1/2-\mu$ it is known, cf. \cite{Lesc}, that 
\begin{equation*}
 \mathrm{dim}\,\calD_{\max}(\wh{A})\big/\calK^{\mu,\mu}(\rz_+)= 
 \mathrm{dim}\,\calD_{\max}({A})\big/\calK^{\mu,\mu}(\rz_+)<\infty  
\end{equation*} 
and that there exist finite-dimensional spaces $\wh\calE$ and $\calE$ of smooth 
functions on $\rz_+$ such that 
\begin{align*}
 \scrD_{\max}(\wh{A})=\calK^{\mu,\mu}(\rz_+)\oplus\wh\frakE, \qquad
 \scrD_{\max}(A)=\calK^{\mu,\mu}(\rz_+)\oplus\frakE.
\end{align*}
Obviously the above equality of dimensions means that 
$\mathrm{dim}\,\wh\frakE=\mathrm{dim}\,\frakE$. In case $f_0^{-1}$ has a pole on the line 
$\re z=1/2-\mu$, the above remains true upon replacing $\calK^{\mu,\mu}(\rz_+)$ by 
 $$\calD_{\min}(A):=\calD_{\max}(A)\cap
     \mathop{\mbox{\LARGE$\cap$}}_{\eps>0}\calK^{\mu,\mu-\eps}(\rz_+).$$ 

We shall now describe a constructive method how to determine the spaces $\wh\frakE$ and 
$\frakE$, which at the same time establishes a canonical 1-1-correspondence between the 
subspaces of $\wh\frakE$ and $\frakE$. This correspondence coincides with that found in 
\cite{GKM2} and plays an important role in the study of the resolvent of $A$, see 
Remark \ref{thm:resolvent} below. We will use the following notation: 
\begin{equation}\label{eq:sigma}
 \Sigma=\Big\{\sigma\in\cz\st \sigma\text{ is a pole of $f_0^{-1}$ and } 
  1/2-\mu<\re\sigma<1/2\Big\}.
\end{equation}
Let us now describe the maximal domain of the model cone operator.  We let 
$\omega, \omega_0\in\scrC^\infty_0(\rpbar)$ be arbitrary cut-off functions and use the 
Mellin transform 
 $$\wh{u}(z)=\int_0^\infty t^{z} u(t)\frac{dt}{t}.$$

\begin{theorem}
For $\sigma\in\Sigma$ define $G_\sigma^{(0)}:\calK^{0,\mu}(\rz_+)\to\calK^{\infty,0}(\rz_+)$ 
by 
 $$(G_\sigma^{(0)} u)(t)=\omega(t)\int_{|z-\sigma|=\eps} 
   t^{-z}f_0^{-1}(z)\wh{\omega_0 u}(z)\,\dbar z,$$
where $\eps>0$ is so small that there is no other pole of $f_0^{-1}$ having distance to 
$\sigma$ less or equal to $\eps$. Then 
 $$\wh\frakE=\mathop{\mbox{\Large$\oplus$}}_{\sigma\in\Sigma} \wh\frakE_\sigma,\qquad 
     \wh\frakE_\sigma=\mathrm{range}\,G_\sigma^{(0)}.$$
\end{theorem}
This result is well-known and we omit the proof. To decribe the maximal domain of $A$ itself define 
recursively 
\begin{equation}\label{eq:recursion}
 g_0=1,\qquad
 g_\ell=-(T^{-\ell}f^{-1}_0)\smsum_{j=0}^{\ell-1}(T^{-j}f_{\ell-j})g_j,\qquad \ell\in\nz,
\end{equation}
with $T^\rho$, $\rho\in\rz$, acting on meromorphic functions by $(T^\rho f)(z)=f(z+\rho)$. 
The $g_j$ are meromorphic and the recursion is equivalent to 
\begin{equation}\label{eq:recursion2}
 \smsum_{\ell=0}^j(T^{-\ell}f_{j-\ell}){g}_\ell
 =\begin{cases}f_0&\quad: j=0\\ 0&\quad: j\ge 1\end{cases}.
\end{equation}
If $h$ is a meromorphic function, denote by $\Pi_\sigma h$ the principal part of the 
Laurent series in $\sigma$; 
of course if $h$ is holomorphic in $\sigma$ then $\Pi_\sigma h=0$. 

\begin{theorem}\label{thm:theta}
For $\sigma\in\Sigma$ and $\ell\in\nz$ define  
$G_\sigma^{(\ell)}:\calK^{0,\mu}(\rz_+)\to\calK^{\infty,0}(\rz_+)$ by 
 $$(G_\sigma^{(\ell)} u)(t,x)=\omega(t) t^\ell\,
   \int_{|z-\sigma|=\eps} t^{-z}g_\ell(z)\,\Pi_\sigma(f_0^{-1}\wh{\omega_0u})(z)\,\dbar z,$$
as well as 
\begin{equation}\label{eq:gsigma}
 G_\sigma:=\smsum_{\ell=0}^{\mu_\sigma}G_\sigma^{(\ell)},\qquad
 \mu_\sigma=[\re\sigma+\mu-1/2],
\end{equation}
where $[x]$ denotes the integer part of $x\in\rz$. Then 
 $$\frakE=\mathop{\mbox{\Large$\oplus$}}_{\sigma\in\Sigma} \frakE_\sigma,
   \qquad \frakE_\sigma=\mathrm{range}\,G_\sigma.$$
Moreover, the following map is well-defined and an isomorphism$:$
\begin{equation}\label{eq:isom1}
  \theta_\sigma:\frakE_\sigma\longrightarrow\wh{\frakE}_\sigma,\quad 
   G_\sigma(u)\mapsto G^{(0)}_\sigma(u).
\end{equation}
\end{theorem}

The theorem is a consequence of Propositions \ref{prop:max} and \ref{prop:basis} 
below. Before we state and prove these, let us remark that the maps $\theta_\sigma$ induce 
an isomorphism 
 $$\theta:\frakE=\mathop{\mbox{\Large$\oplus$}}\limits_{\sigma\in\Sigma}\frakE_\sigma
     \longrightarrow
     \wh{\frakE}=\mathop{\mbox{\Large$\oplus$}}\limits_{\sigma\in\Sigma}\wh{\frakE}_\sigma,$$
which yields the above mentioned 1-1-correspondence of subspaces of $\frakE$ and $\wh\frakE$, 
respectively. This correpondence is important in view of the following result which is due to 
\cite{GKM2}. 

\begin{remark}\label{thm:resolvent}
Let $\underline{A}$ denote the closed operator in $L^2(\rz_+)$ acting as $A$ on the domain 
$\calK^{\mu,\mu}(\rz_+)\oplus\underline{\frakE}$,  where $\underline{\frakE}$ is a subspace of $\frakE$. 
Moreover, let $\underline{\wh{A}}$ be defined by $\wh{A}$ on the domain 
$\underline{\wh{\frakE}}:=\theta(\underline{\frakE})$. Then a ray $\Gamma=e^{i\varphi}\rz_+$ 
in the complex plane is a ray of minimal 
growth for $\underline{A}$ if and only if it is one for $\underline{\wh{A}}$. 
\end{remark}

\begin{proposition}\label{prop:max}
$\frakE_\sigma$ is a subspace of $\calD_{\max}(A)$. 
\end{proposition}
\begin{proof}
By construction $\frakE_\sigma$ is contained in $\calK^{\infty,0}(\rz_+)$. 
Now let $v=G_\sigma(u)$ with $u\in\calK^{0,\mu}(\rz_+)$. We show that $Av$ belongs to $L^2(\rz_+)$. 

First assume that all the integrands appearing in the explicit expression of 
$G_\sigma(u)$ are holomorphic in 
$\mathcal{Z}\setminus\{\sigma\}$ for a vertical strip $\mathcal{Z}=\{z\in\cz\st |\re(\sigma-z)|\le\eps\}$ 
$($the general case we shall treat below$)$. Then we can replace in the explicit expression 
of $G_\sigma(u)$ the integrals $\displaystyle\int_{|z-\sigma|=\eps}$ by the difference 
$\displaystyle\int_{\re z=\re\sigma+\eps}-\int_{\re z=\re\sigma-\eps}$, 
where the lines are oriented upwards. Note that each of the latter two integrals is 
an inverse Mellin transform of the corresponding integrand. 
Now we decompose the operator $A$ as 
 $$A=\omega_1 t^{-\mu}\sum\limits_{j=0}^{\mu-1} t^j f_j(-t\partial_t)+R,$$
where $R$ is a remainder that maps $\calK^{\mu,0}(\rz_+)$ to $L^2(\rz_+)$, and 
$\omega_1$ is chosen in such a way that $\omega\omega_1=\omega_1$. Observing that 
$\omega G^{(\ell)}_\sigma$ maps into $\calK^{\infty,\mu+\ell-\mu_\sigma-\delta}(\rz_+)$ 
for arbitrarily small $\delta>0$, we see that $Av\in L^2(\rz_+)$ provided 
 $$\omega_1\sum_{j=0}^{\mu_\sigma}
   \sum_{l=0}^{\mu_\sigma-j}t^j f_j(-t\partial_t)G^{(\ell)}_\sigma(u)
   \;\in\; \calK^{0,\mu}(\rz_+).$$
By rearranging the summation this is equivalent to 
 $$\omega_1\sum_{k=0}^{\mu_\sigma}t^k\sum_{\ell=0}^{k}(T^{-\ell}f_{k-\ell})(-t\partial_t)
   (t^{-\ell}G^{(\ell)}_\sigma(u))
   \;\in\; \calK^{0,\mu}(\rz_+);$$
we also have used the Mellin operator identity 
$f(-t\partial_t)t^{-\rho}=t^{-\rho}(T^{\rho}f)(-t\partial_t)$. 
The contribution of the inner sum (that over $\ell$) equals, for each $k$,  
 $$\Big(\int_{\re z=\re\sigma+\eps}-\int_{\re z=\re\sigma-\eps}\Big)
   \sum_{\ell=0}^{k}(T^{-\ell}f_{k-\ell})(z)g_\ell(z)
   \Pi_\sigma(f_0^{-1}\wh{\omega_0u})(z)\,\dbar z.$$
However this equals zero as each integrand is holomorphic in the strip $S$, since by 
definition of the $g_j$'s it actually coincides with 
$\delta_{k0}f_0(z)\Pi_\sigma(f_0^{-1}\wh{\omega_0u})(z)$. 

It remains to treat the case where the integrands may have poles in $\mathcal{Z}$ other than $\sigma$. 
However, in this case one takes a function $\varphi\in\cicomp(\rz_+)$ such that 
$\psi:=\calM\varphi$ vanishes to high order in all poles in $\mathcal{Z}$, except for $\sigma$ where 
$1-\psi$ vanishes of high order (cf. Lemma \ref{lem:meromorph}, below). 
Then replace the $g_j$ by $g_j\psi$. This does not effect 
the operator $G_\sigma$, and one can proceed as before, finishing with the expression 
$\delta_{k0}\psi(z)f_0(z)\Pi_\sigma(f_0^{-1}\wh{\omega_0u})(z)$ which is holomorphic in 
the strip $\mathcal{Z}$, again. 
\end{proof}

\begin{proposition}\label{prop:basis}
Let $u,v\in\calK^{\mu,\mu}(\rz_+)$. Then $G_\sigma(u)=G_\sigma(v)$ if and only if 
$G^{(0)}_\sigma(u)=G^{(0)}_\sigma(v)$. In particular, $\frakE_\sigma$ has the same 
dimension as $\wh{\frakE}_\sigma$. 
\end{proposition}
\begin{proof}
Set $w=u-v$. Let first $G^{(0)}_\sigma(w)=0$. Write 
 $$\Pi_\sigma(f_0^{-1}\wh{\omega_0w})(z)=\sum_{\ell=0}^n c_\ell(z-\sigma)^{-(\ell+1)}$$
with certain coefficients $c_\ell\in\cz$. Since 
 $$t^{-z}=\exp(-z\log t)=t^{-\sigma}\smsum_{k=0}^\infty \frac{(-\log t)^k}{k!}(z-\sigma)^k,$$
we see that the residue of 
$t^{-z}\Pi_\sigma(f_0^{-1}\wh{\omega_0w})(z)$ in $z=\sigma$ coincides with 
 $$t^{-\sigma}\sum_{\ell=0}^n \frac{(-1)^\ell}{\ell!}c_\ell\log^{\ell}t.$$
Thus it follows that $G^{(0)}_\sigma(w)=0$ if and only if all $c_\ell=0$, i.e., if and only if 
$\Pi_\sigma (f_0^{-1}\wh{\omega_0w})\equiv0$. This obviously implies $G_\sigma(w)=0$. 
Vice versa, $G_\sigma(w)=0$ implies that 
$G^{(0)}_\sigma(w)=-\sum\limits_{\ell=1}^{\mu_\sigma}G_\sigma^{(\ell)}(w)$. 
However, by construction 
 $$\mathrm{range}\,G^{(0)}_\sigma\cap 
   \mathrm{range}\,\smsum_{\ell=1}^{\mu_\sigma}G_\sigma^{(\ell)}=\{0\},$$
hence $G^{(0)}_\sigma(w)=0$. The second statement then follows, since 
$G^{(0)}_\sigma(u_1),\ldots,G^{(0)}_\sigma(u_m)$ 
are linearly independent if and only if  $G_\sigma(u_1),\ldots,G_\sigma(u_m)$ are, 
as can be seen by considering linear combinations.  
\end{proof}

\subsection{Explicit formulae for the domains}\label{sec:3.3}

The above defined domains can be characterized explicitly using the residue theorem. 
Before doing so let us state the following simple fact.  

\begin{lemma}\label{lem:meromorph}
Let $H$ be a Hilbert space and $\eps>0$ arbitrary. To any given pairwise different points 
$\sigma_0,\ldots,\sigma_k\in\cz$, non-negative integers $n_0,\ldots,n_k$, and elements 
$x_0,\ldots,x_{n_0}\in H$ there exists a function $u\in\scrC^\infty_0((0,\eps),X)$ such that 
the Mellin transform $\wh{u}$ of $u$ has zeros of order $n_j$ in the points $p_j$ for 
$j=1,\ldots,k$ and $d_z^j\wh{u}(p_0)/j!=x_j$ for $j=0,\ldots,n_0$. 
\end{lemma}

Now let us evaluate $G^{(0)}_\sigma u$ for some $\sigma\in\Sigma$. To this end let 
\begin{equation*}
  f_0(z)^{-1}\sim \sum_{k=0}^{n_\sigma}r_{\sigma,k}(z-\sigma)^{-(k+1)},\qquad 
  r_{\sigma,n_\sigma}\not=0,
\end{equation*}
denote the principal part of the Laurent expansion of $f_0^{-1}$. Then, by the residue theorem 
(see the proof of Proposition \ref{prop:basis}),   
\begin{equation}\label{gexplicit}
 (G^{(0)}_\sigma u)(t)=\omega(t)t^{-\sigma}\sum^{n_\sigma}_{j=0}
 \zeta_{\sigma,j}(u)\, \log^j t,
\end{equation}
where the coefficients $\zeta_{\sigma,\ell}(u)$ are computed by 
\begin{equation}\label{eq:zeta}
 \zeta_{\sigma,j}(u)=\frac{(-1)^{j}}{j!}\sum_{k=j}^{n_\sigma}
 r_{\sigma,k}\delta_{\sigma,k-j}(u),\qquad 
 \delta_{\sigma,i}(u)=d_z^{i}\wh{\omega_0 u}(\sigma)/i!.
\end{equation}
Writing $\zeta_\sigma(u)=(\zeta_{\sigma,0}(u),\ldots,\zeta_{\sigma,n_\sigma}(u))$ and 
$\delta_\sigma(u)=(\delta_{\sigma,0}(u),\ldots,\delta_{\sigma,n_\sigma}(u))$, 
in matrix notation this reads as 
 $$\zeta_\sigma(u)=B_\sigma \delta_\sigma(u),\qquad 
   B_\sigma=(b_{\sigma,jk})_{0\le j,k\le n_\sigma},$$
where the coefficients of $B_\sigma$ are given by 
 $$b_{\sigma,jk}=\begin{cases}
      (-1)^j r_{\sigma,j+k}/j!&\quad: j+k\le n\\
      0 &\quad: j+k>n
      \end{cases}$$
(this formula also holds true for a general cross-section $X$, where now the $r_{\sigma,k}$ are 
smoothing operators on $X$). As the left-upper triangular matrix $B_\sigma$ is invertible and, 
by the above lemma, $\delta_\sigma(u)$ runs through all of $\cz^{n_\sigma+1}$ when $u$ 
varies over $\scrC^\infty_0(\rz_+)$, we conclude the following. 

\begin{proposition}\label{prop:hatE}
If $n_\sigma$ is the multiplicity of the pole $\sigma\in\Sigma$ of $f_0^{-1}$ then  
\begin{equation*}
 \wh\frakE_\sigma=\bigg\{\omega(t)\sum_{j=0}^{n_\sigma}
  a_jt^{-\sigma}\log^j t \mbox{\boldmath$\;\Big|\;$\unboldmath} 
  a\in\cz^{n_\sigma+1}\bigg\}=\calE_{\wh{S}_\sigma}\cong \cz^{n_\sigma+1},
\end{equation*}
with asymptotic type $\wh{S}_\sigma=\{(\sigma,n_\sigma)\}$. 
\end{proposition}

Finding the explicit representation of the operators $G_\sigma^{(\ell)}$ works along the 
same lines. If we write 
 $$g_\ell(z)\sim\sum_{k=-N_\sigma^{(\ell)}}^\infty g^{(\ell)}_{\sigma,k}(z-\sigma)^k,
   \qquad N_\sigma^{(\ell)}\ge 0,$$
for the Laurent series of $g_\ell$ around $\sigma$ then a direct computation using the 
residue theorem shows that  
\begin{equation}\label{glexplicit}
 (G_\sigma^{(\ell)}u)(t)=\omega(t)t^{-\sigma+\ell}
 \sum_{j=0}^{N_\sigma^{(\ell)}+n_\sigma}
 \Big(\sum_{k=\max(0,j-N_\sigma^{(\ell)})}^{n_\sigma}
 \frac{(-1)^k k!}{(-1)^j j!} g^{(\ell)}_{\sigma,k-j}\zeta_{\sigma,k}(u)\Big)\log^j t
\end{equation}
with the $\zeta_{\sigma,k}(u)$ as introduced in \eqref{eq:zeta}. 
Now denote by $\skp{\cdot}{\cdot}_{n_\sigma}$ the inner product of $\cz^{n_\sigma+1}$ and 
by $e_k$ the $k$-th unit vector. 
If we then define the vectors $x^{(\ell)}_{\sigma,j}\in \cz^{n_\sigma+1}$, 
$j=0,\ldots,N_\sigma^{(\ell)}+n_\sigma$, by 
 $$\left\langle e_k,x^{(\ell)}_{\sigma,j}\right\rangle_{n_\sigma}
   =\frac{(-1)^j k!}{(-1)^k j!}g^{(\ell)}_{\sigma,k-j},\qquad k=0,\ldots,n_\sigma,$$
in case $0\le j\le N_\sigma^{(\ell)}$ and by 
 $$\left\langle e_k,x^{(\ell)}_{\sigma,j}\right\rangle_{n_\sigma}
   =\begin{cases} 
     0 &\quad: 0\le k\le j-N_\sigma^{(\ell)}-1\\
     \frac{(-1)^j k!}{(-1)^k j!}g^{(\ell)}_{\sigma,k-j} 
       &\quad: j-N_\sigma^{(\ell)}\le k\le n_\sigma
    \end{cases}
    ,\qquad k=0,\ldots,n_\sigma,$$
provided $N_\sigma^{(\ell)}+1\le j\le N_\sigma^{(\ell)}+n_\sigma$, 
then we can write 
\begin{equation}\label{glexplicit2}
 (G^{(\ell)}_\sigma u)(t)=\omega(t)t^{-\sigma+\ell}
 \sum_{j=0}^{N_\sigma^{(\ell)}+n_\sigma}
 \left\langle\zeta_\sigma(u),x^{(\ell)}_{\sigma,j}\right\rangle_{n_\sigma}\log^j t.
\end{equation}

Again using that $u\mapsto\zeta(u):\scrC^\infty_0(\rz_+)\to \cz^{n_\sigma+1}$ is surjective, 
we obtain the following description of the spaces $\frakE_\sigma$. 

\begin{proposition}\label{prop:E}
With the previously introduced notation
\begin{equation*}
 \frakE_\sigma
 =\bigg\{\omega(t)\sum_{\ell=0}^{\mu_\sigma}
 \sum_{j=0}^{N_\sigma^{(\ell)}+n_\sigma}
 \left\langle a,x^{(\ell)}_{\sigma,j}\right\rangle_{n_\sigma}
 t^{-\sigma+\ell}\log^j t
 \mbox{\boldmath$\;\Big|\;$\unboldmath} 
  a\in\cz^{n_\sigma+1}\bigg\}\cong \cz^{n_\sigma+1}. 
\end{equation*}
\end{proposition}

The latter two Propositions \ref{prop:hatE} and \ref{prop:E} obviously yield 
an explicit representation of the isomorphism $\theta_\sigma:\wh\frakE_\sigma\to\frakE_\sigma$ 
from \eqref{eq:isom1}, namely  
\begin{equation}\label{eq:isom}
 \theta_\sigma^{-1}\Big(\omega(t)\sum_{j=0}^{n_\sigma}a_jt^{-\sigma}\log^j t\Big)
 =\omega(t)\sum_{\ell=0}^{\mu_\sigma}\sum_{j=0}^{N_\sigma^{(\ell)}+n_\sigma}
    \left\langle a,x^{(\ell)}_{\sigma,j}\right\rangle_{n_\sigma}
    t^{-\sigma+\ell}\log^j t.
\end{equation}
Note that $N_\sigma^{(0)}=0$ and $x^{(0)}_{\sigma,j}=e_j$ for $j=0,\ldots,n_\sigma$, 
so the summand on the right-hand side for $\ell=0$ is just the function from the 
left-hand side. 

We have seen in Proposition \ref{prop:hatE} that $\wh\frakE_\sigma$ equals  
$\calE_{\wh{S}_\sigma}$ with the asymptotic type $\wh{S}_\sigma=\{(\sigma,n_\sigma)\}$. 
However, $\frakE=\oplus_{\sigma\in\Sigma}\frakE_\sigma$ in general does not coincide 
with a space $\calE_{S}$ for any asymptotic type $S$. Choosing $S$ suitably, $\frakE$ 
will be a subspace of $\calE_S$ and we can construct a canonical projection of $\calE_{S}$ 
onto $\frakE$. This we shall describe in the following remark. 

\begin{remark}\label{rem:type}
With the previously introduced notation let 
$N\ge\max\limits_{\sigma\in\Sigma}\max\limits_{\ell=0}^{\mu_\sigma}N_\sigma^{(\ell)}$ 
and define the asymptotic type 
 $$S=\{(\sigma-\ell,N)\st \sigma\in\Sigma,\,\ell=0,\ldots,\mu_\sigma\}.$$
Then $\frakE$ is a subspace of $\calE_{S}$. For any $\sigma\in\Sigma$ there is an obvious 
projection $\wh{P}_\sigma$ of $\calE_{S}$ onto $\calE_{\wh{S}_\sigma}$, where 
$\wh{S}_\sigma=\{(\sigma,n_\sigma)\}$. Now define 
 $$\iota:\cz\to\rz,\qquad \iota(x+iy)=y.$$ 
Obviously $\iota(\Sigma)=\{y_1,\ldots,y_k\}$ is finite. 
Write $\Sigma_i=\Sigma\cap\iota^{-1}(y_i)$ 
and order the elements $\sigma_{i0},\sigma_{i1},\ldots,\sigma_{ik_i}$ of $\Sigma_i$ by 
decreasing real parts, i.e., $\re\sigma_{ij}>\re\sigma_{i(j+1)}$. We define 
a projection $\pi_i$ of $\calE_S$ onto $\oplus_{\sigma\in\Sigma_i}\frakE_\sigma$ in 
the following way: For $u\in\calE_S$ let $u_0=u$ and then 
 $$u_{j+1}:=u_{j}-\theta_{\sigma_{ij}}^{-1}\big(\wh{P}_{\sigma_{ij}} u_{j}\big), 
   \qquad j=0,\ldots,k_i-1,$$
using the isomorphisms from \eqref{eq:isom}. Define $\pi_iu=u_{k_i}$. The desired projection 
of $\calE_S$ onto $\frakE$ is then $\pi:=\pi_1+\ldots+\pi_k$.
\end{remark}

\section{Edge-degenerate Differential Operators on a Half-space}\label{sec:4}

We shall now use the results derived in the previous section for the description of 
natural, in a certain sense maximal, domains for edge-degenerate differential operators. However, first 
we provide some background material concerning pseudodifferential operators with operator-valued 
symbols that we shall need later on. 

\subsection{Pseudodifferential operators with operator-valued symbols}\label{sec:4.1}

Let $H$ be a Hilbert space. A group action on $H$ is a function $\kappa_\lambda:[0,\infty)\to\scrL(H)$, 
the bounded operators on $H$, such that 
\begin{itemize}
 \item[(1)] $\kappa_\lambda\kappa_\rho=\kappa_{\lambda\rho}$, $\kappa_1=\mathrm{id}$, 
 \item[(2)] $\kappa_\lambda x\xrightarrow{\lambda\to 1}x$ for any $x\in H$. 
\end{itemize}
We think $H$ to be equipped with such a group action. We shall denote by $\eta\mapsto[\eta]$ a 
positive smooth function on $\rz^q$ which coincides with $|\eta|$ outside the unit ball.  

\begin{definition}
For $s\in\rz$ we let $\calW^s(\rz^q,H)$ denote the closure of $\scrS(\rz^q,H)$ with respect to the norm 
 $$\|u\|=\Big(\int_{\rz^q}[\eta]^{2s}\|\kappa^{-1}(\eta)\scrF u(\eta)\|_H^2\Big)^{1/2},$$
where we define $\kappa(\eta):=\kappa_{[\eta]}$ and $\scrF$ denotes the Fourier transform. 
This is a Hilbert space. If the group action is trivial, $\kappa\equiv1$, we write $H^s(\rz^q,H)$. 
\end{definition}

The spaces $\calW^s(\rz^q,H)$ are called abstract edge Sobolev spaces. Note that the operator 
$L=\scrF^{-1}\,\kappa^{-1}(\eta)\,\scrF$ induces a canonical isometric isomorphism between 
$\calW^s(\rz^q,H)$ and $H^s(\rz^q,H)$. Pseudodifferential operators 
in this set-up are based on operator-valued symbols. 

\begin{definition}
Let $H$ and $\wt{H}$ be two Hilbert spaces with group action and $\mu\in\rz$. 
Then $S^\mu(\rz^q\times\rz^q;H,\wt{H})$ is the space of all smooth functions 
$a(y,\eta):\rz^q\times\rz^q\to\scrL(H,\wt{H})$ 
satisfying estimates 
 $$\left\|\wt{\kappa}^{-1}(\eta)\big(D^\alpha_\eta D^\beta_y a(y,\eta)\big)
    \kappa(\eta)\right\|_{\scrL(H,\wt{H})}\le 
    C_{\alpha\beta} [\eta]^{\mu-|\alpha|}$$
for any multi-indices $\alpha$ and $\beta$. The associated pseudodifferential operator is denoted by $a(y,D)$. 
\end{definition}

The operator $a(y,D)$ is defined analogously to the case where $H=\wt{H}=\cz$ and is initially a map from 
$\scrS(\rz^q,H)$ to $\scrS(\rz^q,\wt{H})$. It can be shown, cf. \cite{Schu1}, \cite{Seil2},  that it extends to continuous 
maps 
\begin{equation}\label{eq:pseudo}
 a(y,D):\calW^s(\rz^q,H)\lra \calW^{s-\mu}(\rz^q,\wt{H}),\qquad s\in\rz, 
\end{equation}
if $\mu$ is the order of $a(y,\eta)$. 

A function $p(y,\eta):\rz^q\times(\rz^q\setminus\{0\})\to\scrL(H,\wt{H})$ is called twisted homogeneous of degree 
$d$, if the identity 
\begin{equation}\label{eq:twisted}
 p(y,\lambda\eta)=\lambda^d \, \wt{\kappa}_\lambda \, p(y,\eta)\, \kappa_\lambda^{-1}
\end{equation}
holds true for any $(y,\eta)$ and any positive $\lambda$. The space of such twisted homogeneous functions 
we shall denote by $S^{(d)}(\rz^q\times\rz^q;H,\wt{H})$. 

\begin{definition}
A symbol $a\in S^\mu(\rz^q\times\rz^q;H,\wt{H})$ is called classical if there exists a sequence of twisted 
homogeneous symbols $a^{(\mu-j)}(y,\eta)$ of degree $\mu-j$ such that 
 $$a(y,\eta)-\sum_{j=0}^{N-1} \chi(\eta)a^{(\mu-j)}(y,\eta) \in S^{\mu-N}(\rz^q\times\rz^q;H,\wt{H})$$
for any $N\in\nz$, where $\chi(\eta)$ denotes a zero excision function. 
The space of such symbols shall be denoted by $S^\mu_{c\ell}(\rz^q\times\rz^q;H,\wt{H})$. We set 
 $$\sigma_\wedge^\mu(a)(y,\eta)=a^{(\mu)}(y,\eta)$$
and call this function the homogeneous principal symbol of $a$. 
\end{definition}

Occasionally we will consider $H$ and $\wt{H}$ with the trivial group action $\kappa\equiv1$. If this is not clear from
the context, we point this out by writing $S^\mu_{c\ell}(\rz^q\times\rz^q;H,\wt{H})_{(1)}$. 

In our application we will deal with Hilbert spaces that are function or distribution spaces on $\rz_+$. They 
will be always equipped with the `standard group-action' which is defined by 
\begin{equation}\label{eq:group}
 (\kappa_\lambda u)(t)=\lambda^{1/2}u(\lambda t),
\end{equation}
i.e., it is the dilation group we already have seen in Section \ref{sec:2}. We assume this from now on and do not indicate 
it furthermore. 

\begin{example}
For any $s\in\rz$ the spaces $H^s(\rz^q\times\rz_+)$ and $\calW^s(\rz^q,H^s(\rz_+))$ are naturally 
isomorphic, cf. Section 3.1.1 in \cite{Schu1}. 
\end{example}

The previous example motivates the following definition. 

\begin{definition}
For $s,\gamma\in\rz$ we define the `edge Sobolev spaces' 
 $$\calW^{s,\gamma}(\rz^q\times\rz_+):=\calW^s(\rz^q,\calK^{s,\gamma}(\rz_+)),$$
and subspaces 
 $$\calW^{s,\gamma}(\rz^q\times\rz_+)_S:=\calW^s(\rz^q,\calK^{s,\gamma-\theta}(\rz_+)\oplus\calE_S),$$
where $S\in\mathrm{As}(\gamma,\theta)$ is an asypmtotic type, cf. Definition {\rm\ref{def:coneasymp}}. 
\end{definition}

Using $L$ from above we have an isomorphism from $\calW^{s,\gamma}(\rz^q\times\rz_+)_S$ to 
 $$H^s(\rz^q,\calK^{s,\gamma-\theta}(\rz_+)\oplus\calE_S)=
     H^s(\rz^q,\calK^{s,\gamma-\theta}(\rz_+))\oplus H^s(\rz^q,\calE_S),$$
we  define 
\begin{equation}\label{eq:abc}
 \calV^s(\rz^q,\calE_S)=L^{-1}H^s(\rz^q,\calE_S). 
\end{equation}
This is a closed subspace of $\calW^{s,\gamma}(\rz^q\times\rz_+)_S$. Note that $\calE_S$ alone is not 
invariant under the group action due to the cut-off function $\omega$ involved in its definition, but 
$\calK^{s,\gamma-\theta}(\rz_+)\oplus\calE_S$ is.  

\subsection{Construction of the natural domain}\label{sec:4.2}

Consider an edge-degenerate differential operator $A$ on the half-space 
$\Omega=\rz^q\times\rz_+$ with $y$-independent coefficients (as before, we could allow a cone 
$\rz_+\times X$ with non trivial base). Near the boundary let 
 $$A=t^{-\mu}\sum_{j+|\alpha|=0}^\mu a_{j\alpha}(t)(tD_y)^\alpha(-t\partial_t)^j.$$
We assume that $(1-\omega)(t)A$ maps $\calW^{\mu,0}(\Omega)$ into 
$L^2(\Omega)=\calW^{0,0}(\Omega)$ (i.e., the coefficients behave well as $t\to\infty$) 
and that 
 $$f_0(z):=\sum_{j=0}^{\mu} a_{j0}(0)z^j$$
is meromorphically invertible and has no pole on the line $\re z=1/2-\mu$ 
(the meromorphic invertibility is automatically 
satisfied for suitably elliptic operators, and also holds in case of non trivial $X$). 
We will define a natural domain $\scrD_{\max}(A)\subset \calW^{\mu,0}(\Omega)$ such that  
$A:\scrD_{\max}(A)\lra L^2(\Omega)$ (we also could consider $\calW^{s,\gamma}(\Omega)$ for 
arbitrary $s$ and $\gamma$ but for convenience we take $s=\gamma=0$). 
By abuse of notation this domain, in general, does not yield the maximal closed extension of $A$ in the 
functional analytic sense. 

Denoting the Taylor expansion of the coefficient functions $a_{j\alpha}$ by 
 $$a_{j\alpha}(t)\sim\sum_{k=0}^\infty a_{j\alpha}^{(k)}t^k$$
we define the truncated operator $A_{tr}$ by
 $$A_{tr}=t^{-\mu}\sum_{\ell=0}^\mu t^\ell\sum_{k+|\alpha|=\ell}\sum_{j=0}^{\mu-|\alpha|} 
   a_{j\alpha}^{(k)}D_y^\alpha(-t\partial_t)^j,$$
and then set $\wt{A}:= L\circ A_{tr}\circ L^{-1}$. This operator can be viewed as a 
pseudodifferential operator with operator-valued symbol $\wt{\fraka}(\eta)$ which is  
 $$\wt{\fraka}(\eta)
   =t^{-\mu}\sum_{\ell=0}^\mu t^\ell \wt{f}_\ell(-t\partial_t,\eta),\qquad 
   \wt{f}_\ell(z,\eta)=[\eta]^{\mu-\ell}\sum_{k+|\alpha|=\ell}
   \sum_{j=0}^{\mu-|\alpha|} a_{j\alpha}^{(k)}\eta^\alpha z^j.$$
Of course, $\wt{\fraka}(\eta)$ is for each $\eta$ a cone differential operator and 
the $\wt{f}_\ell(z,\eta)$ are the corresponding conormal symbols. Note that 
$\wt{f}_\ell(z,\eta)$ is a classical symbol of order $\mu$ in $\eta$, and that 
$\wt{f}_0(z,\eta)=[\eta]^\mu f_0(z)$. Again we shall use the notation  
 $$\Sigma=\Big\{\sigma\in\cz\st \sigma\text{ is a pole of $f_0^{-1}$ and } 
   1/2-\mu<\re\sigma<1/2\Big\},$$
write $n_\sigma$ for the multiplicity of the pole $\sigma$ of $f_0^{-1}$, and set 
$\mu_\sigma=[\re\sigma+\mu-1/2]$. We use the asymptotic types 
 $$S_\sigma=\{(\sigma-\ell,m_\sigma)\st 0\le \ell\le \mu_\sigma\},\qquad 
     m_\sigma=\max_{\ell=0}^{\mu_\sigma}N_\sigma^{(\ell)},$$ 
cf. Remark \ref{rem:type}. 

Define recursively the functions $\wt{g}_\ell(z,\eta)$ as in \eqref{eq:recursion},  
replacing the $f_j(z)$ by $\wt{f}_j(z,\eta)$. Let $\omega,\omega_0\in\scrC^\infty(\rz_+)$ 
be arbitrary fixed cut-off functions. Then the expressions 
\begin{align*}
 [\wt\frakg^{(0)}_\sigma(\eta)u](t)
   &=\omega(t)\int_{|z-\sigma|=\eps}t^{-z}\wt{f}_0(z,\eta)^{-1}\wh{\omega_0 u}(z)\,\dbar z\\
 [\wt\frakg^{(\ell)}_\sigma(\eta)u](t)
   &=\omega(t)t^\ell\int_{|z-\sigma| =\eps}t^{-z}\wt{g}_\ell(z,\eta)
     \Pi_\sigma\big(\wt{f}_0(z,\eta)^{-1}\wh{\omega_0 u}(z)\big)\,\dbar z, 
\end{align*}
define operator-valued symbols 
\begin{align}\label{eq:green1}
 \wt\frakg_\sigma^{(\ell)}(\eta)\in 
  S^{-\mu}_{c\ell}(\rz^q;\calK^{0,\mu}(\rz_+),\calE_{S_\sigma})_{(1)}.
\end{align}

\begin{theorem}\label{thm:domain1}
With the symbols defined in \eqref{eq:green1} let 
$\wt\frakg_\sigma(\eta)
=\sum\limits_{\ell=0}^{\mu_\sigma}\wt\frakg^{(\ell)}_\sigma(\eta)$.  Then 
 $$\mathrm{range}\,\wt\frakg_\sigma(D)
    =\mathrm{range}\Big(\wt\frakg_\sigma(D):L^2(\rz^q,\calK^{0,\mu}(\rz_+))
    \lra H^\mu(\rz^q,\calE_{S_\sigma})\Big)$$
is a closed subspace of $H^\mu(\rz^q,\calE_{S_\sigma})$ which is mapped by $\wt{A}$ 
into $L^2(\Omega)$.  
\end{theorem}
\begin{proof}
The closedness of $\mathrm{range}\,\wt\frakg_\sigma(D)$ we shall prove after Theorem 
\ref{thm:edgedom} below, where we derive a more explicit description. The  
mapping property of $\wt{A}$ follows from construction (the details are along the lines 
of the proof of Proposition \ref{prop:max} for the case of cone operators). We omit the details. 
\end{proof}

\subsection{The principal edge symbol}\label{sec:4.3}

The principal edge symbol of $A$ is, by definition, the function 
 $$\sigma^\mu_\wedge(A)(\eta)=
     t^{-\mu}\sum_{j+|\alpha|=0}^\mu 
      a_{j\alpha}(0)(t\eta)^\alpha(-t\partial_t)^j,\qquad\eta\not=0;$$
note that the coefficients $a_{j\alpha}$ are `frozen' in $t=0$. The principal edge symbol 
is $($formally$)$ twisted homogeneous of degree $\mu$ and pointwise, for any $\eta$, 
a cone differential operator on $\rz_+$ of which we assume that it is elliptic in the sense described in 
Section \ref{sec:3} (this is not a restriction, since this is always holds for elliptic edge-degenerate operators). 
We now can apply a procedure analogous to the one in the previous section. It is a bit simpler, since we 
do not have to apply a Taylor expansion to the 
coefficients.  First we define 
 $$\bar{f}_\ell(z,\eta)=|\eta|^{\mu-\ell}\sum_{|\alpha|=\ell}\sum_{j=0}^{\mu-|\alpha|} 
   a_{j\alpha}(0)\eta^\alpha z^j,\qquad \eta\not=0$$
and then recursively functions $\bar{g}_\ell(z,\eta)$ as in \eqref{eq:recursion}, replacing 
the $f_j(z)$ by $\bar{f}_j(z,\eta)$. We define homogeneous functions 
 $$\bar\frakg_\sigma^{(\ell)}(\eta)\in 
     S^{(-\mu)}_{c\ell}(\rz^q;\calK^{0,\mu}(\rz_+),\calE_{{S}_\sigma})_{(1)}$$
by the expressions 
\begin{align*}
 [\bar\frakg^{(0)}_\sigma(\eta)u](t)
   &=\omega(t)\int_{|z-\sigma|=\eps}t^{-z}\bar{f}_0(z,\eta)^{-1}\wh{\omega_0 u}(z)\,\dbar z\\
 [\bar\frakg^{(\ell)}_\sigma(\eta)u](t)
   &=\omega(t)t^\ell\int_{|z-\sigma| =\eps}t^{-z}\bar{g}_\ell(z,\eta)
     \Pi_\sigma\big(\bar{f}_0(z,\eta)^{-1}\wh{\omega_0 u}(z)\big)\,\dbar z   
\end{align*}
which are defined for $\eta\not=0$. We also set $\bar\frakg_\sigma(\eta)
=\sum\limits_{\ell=0}^{\mu_\sigma}\bar\frakg^{(\ell)}_\sigma(\eta)$. 
The following observation will be important for us; it actually follows directly from the 
construction of the symbols $\wt\frakg^{(\ell)}_\sigma(\eta)$ and $\bar\frakg^{(\ell)}_\sigma(\eta)$. 

\begin{proposition} 
The symbols $\bar\frakg^{(\ell)}_\sigma(\eta)$ and $\bar\frakg_\sigma(\eta)$ are the homogeneous 
principal symbols of $\wt\frakg^{(\ell)}_\sigma(\eta)$ and $\wt\frakg_\sigma(\eta)$, respectively. 
\end{proposition}

It is a consequence of the results from Section \ref{sec:3} for cone differential operators that 
\begin{equation}\label{eq:princ}
 \kappa_{|\eta|}^{-1}\,\sigma^\mu_\wedge(A)(\eta)\,\kappa_{|\eta|}:
 \mathrm{range}\,\bar\frakg_\sigma(\eta)\subset\calE_{{S}_\sigma}\lra L^2(\rz_+).
 \end{equation}
 
\subsection{Explicit form of the domains}\label{sec:4.4}

In Section \ref{sec:3.3} we have found explicit representations of the domains for cone differential operators. 
We can follow the procedure introduced there, keeping track of the additional $\eta$-dependence of all involved symbols. 
Doing so we find symbols 
\begin{equation}\label{eq:xjl}
 \wt{x}_{\sigma,j}^{(\ell)}(\eta)\in S^{-\mu}_{c\ell}(\rz^q;\cz^{n_\sigma},\cz),\qquad 
 \bar{x}_{\sigma,j}^{(\ell)}(\eta)\in S^{(-\mu)}_{c\ell}(\rz^q;\cz^{n_\sigma},\cz),
\end{equation} 
$ j=0,\ldots,N_\sigma^{(\ell)}+n_\sigma$, 
which are determined in terms of the Laurent coefficients of the $g_\ell(z,\eta)$, such that $\bar{x}_{\sigma,j}^{(\ell)}(\eta)$ 
is the homogeneous principal symbol of $\wt{x}_{\sigma,j}^{(\ell)}(\eta)$ and the following result is true. 

\begin{theorem}\label{thm:edgedom}
With the above introduced notation $\mathrm{range}\,\wt{\frakg}_\sigma(D)$ equals 
\begin{equation}\label{eq:edgedom}
    \bigg\{\omega(t)\sum_{\ell=0}^{\mu_\sigma}
           \sum_{j=0}^{N_\sigma^{(\ell)}+n_\sigma}
           \big(\wt{x}^{(\ell)}_{\sigma,j}(D)a\big)\,
           t^{-\sigma+\ell}\log^j t
           \mbox{\boldmath$\;\Big|\;$\unboldmath} 
           a\in L^2(\rz^q,\cz^{n_\sigma+1})\bigg\}. 
\end{equation}
\end{theorem}

Now it is easy to complete the proof of Theorem \ref{thm:domain1}. 

\begin{proof}[Proof of Theorem {\rm\ref{thm:domain1}}]
Let $(u^{(n)})$ be a sequence in $\mathrm{range}\,\wt{\frakg}_\sigma(D)$ that converges in 
$H^\mu(\rz^q,\calE_{S_\sigma})$ to $u$. Write  
 $$u^{(n)}= \omega(t)\sum_{\ell=0}^{\mu_\sigma}\sum_{j=0}^{N_\sigma^{(\ell)}+n_\sigma}
     \big(\wt{x}^{(\ell)}_{\sigma,j}(D)a^{(n)}\big)\,t^{-\sigma+\ell}\log^j t$$
with $a^{(n)}\in L^2(\rz^q,\cz^{n_\sigma+1})$. The convergence of $(u^{(n)})$ is equivalent 
to the convergence of any of the sequences $(\wt{x}^{(\ell)}_{\sigma,j}(D)a^{(n)})$ in 
$H^\mu(\rz^q,\cz)$. However, for $\ell=0$ we have 
$\wt{x}^{(\ell)}_{\sigma,j}(D)a^{(n)}=[D]^\mu a^{(n)}_j$, hence $(a^{(n)})$ converges in 
$L^2(\rz^q,\cz^{n_\sigma+1})$. Denoting the limit by $a$ it follows that $(u^{(n)})$ 
converges to 
 $$ \omega(t)\sum_{\ell=0}^{\mu_\sigma}\sum_{j=0}^{N_\sigma^{(\ell)}+n_\sigma}
       \big(\wt{x}^{(\ell)}_{\sigma,j}(D)a\big)\, t^{-\sigma+\ell}\log^j t,$$
which is an element of  $\mathrm{range}\,\wt{\frakg}_\sigma(D)$. 
\end{proof}  

\subsection{The natural domain of $A$}\label{sec:4.5}

As we have derived $\wt{A}$ from $A$ (actually, from $A_{tr}$) by conjugation with 
the isomorphism $L$ we obtain the natural domain for $A$ by pulling back the above 
constructions under $L$. In detail, we have the follwing: 

\begin{theorem}\label{cor:symbol}
With the symbols introduced in \eqref{eq:green1} define 
 $$\frakg_\sigma^{(\ell)}(\eta) 
   =\kappa(\eta)\,\wt\frakg_\sigma^{(\ell)}(\eta)\,\kappa^{-1}(\eta),\qquad
     \frakg_\sigma(\eta)=\sum_{\ell=0}^{\mu_\sigma}\wt\frakg^{(\ell)}_\sigma(\eta).$$
These are operator-valued symbols,  
 $$\frakg_\sigma^{(\ell)}(\eta),\frakg_\sigma(\eta)\in 
   S^{-\mu}_{c\ell}(\rz^q;\calK^{0,\mu}(\rz_+),\calE_{S_\sigma}),$$
and the range of $\frakg_\sigma(D):\calW^{0,\mu}(\Omega)\to \calW^{\mu,0}(\Omega)$ is a closed subspace of 
$\calV^\mu(\rz^q,\calE_{S_\sigma})$ which is mapped by ${A}$ into $L^2(\Omega)$. The homogeneous principal 
symbols are given by 
 $$\sigma^{-\mu}(\frakg_\sigma^{(\ell)})(\eta) 
    =\kappa_{|\eta|}\,\bar\frakg_\sigma^{(\ell)}(\eta)\,\kappa^{-1}_{|\eta|},\qquad
    \sigma^{-\mu}( \frakg_\sigma)(\eta)
    =\kappa_{|\eta|}\,\bar\frakg_\sigma(\eta)\,\kappa^{-1}_{|\eta|}.$$
\end{theorem}

Due to the previous result and the motivation given in Section \ref{sec:2}, the following definition appears natural: 

\begin{definition}\label{def:domain}
We define the natural domain of $A$ as    
 $$\scrD_{\max}(A)=\calW^{\mu,\mu}(\Omega)\oplus
     \mathop{\mbox{\Large$\oplus$}}_{\sigma\in \Sigma}\mathrm{range}\,\frakg_\sigma(D).$$
\end{definition}

Note that  $\scrD_{\max}(A)$ is contained in $\calW^{\mu,\eps}(\Omega)$ for any 
$0<\eps<\min\limits_{\sigma\in \Sigma}1/2-\re\sigma$.  

\begin{definition}\label{def:domainwedge}
According to \eqref{eq:princ} let us define  
 $$\scrD_{\max}(\sigma^\mu_\wedge(A)(\eta))=\calK^{\mu,\mu}(\rz_+)\oplus
 \mathop{\mbox{\Large$\oplus$}}_{\sigma\in \Sigma}
 \mathrm{range}\,\kappa_{|\eta|}\,\bar\frakg_\sigma(\eta).$$
\end{definition}
 
The following theorem shows the existence of a canonical projection on the natural domain of $A$. 

\begin{theorem}\label{thm:proj2}
Let $S$ be the asymptotic type defined in Remark {\rm\ref{rem:type}}. Then there exists a symbol 
 $$p(\eta)\in S^0_{c\ell}(\rz^q;\calK^{\mu,\mu}(\rz_+)\oplus\calE_S,\calK^{\mu,\mu}(\rz_+)\oplus\calE_S)$$ 
having the following properties: 
\begin{itemize}
 \item[(1)] $p(D)$ is a projection in $\calW^{\mu,\mu}(\Omega)\oplus\calV^\mu(\rz^q,\calE_S)$ having $\scrD_{\max}(A)$ as its range, 
  and $p(D)$ is the identity map on $\calW^{\mu,\mu}(\Omega)$. 
 \item[(2)] $\sigma_\wedge^0(p)(\eta)$ is a projection in $\calK^{\mu,\mu}(\rz_+)\oplus\calE_S$ onto 
  $\scrD_{\max}(\sigma^\mu_\wedge(A)(\eta))$
  and  $\sigma_\wedge^0(p)(\eta)$ is the identity map on $\calK^{\mu,\mu}(\rz_+)$. 
\end{itemize}
\end{theorem}
\begin{proof}
The proof is a parameter-dependent variant of the procedure described in Remark \ref{rem:type}. 
First define isomorphisms $\theta_\sigma(\eta)$ as in \eqref{eq:isom}, replacing ${x}_{\sigma,j}^{(\ell)}$ by 
$[\eta]^\mu\wt{x}_{\sigma,j}^{(\ell)}(\eta)$. This yields a projection $\pi(\eta)$ as before. 
Let $\pi^{\prime}(\eta)$ denote the extension of $\pi(\eta)$ by 1 to $\calK^{\mu,\mu}(\rz_+)\oplus\calE_S$. 
Then $p(\eta):=\kappa(\eta)\,\pi^{\prime}(\eta)\,\kappa^{-1}(\eta)$ has the desired properties (recall that 
$\bar{x}_{\sigma,j}^{(\ell)}(\eta)$ is the homogeneous principal symbol of $\wt{x}_{\sigma,j}^{(\ell)}(\eta)$).  
\end{proof}

Referring to the terminology used in \cite{KSS} we call $p(D)$ a `singular' projection, indicating that it is the identity on 
$\calW^{\mu,\mu}(\Omega)$ and acts non trivially only on $\calV^\mu(\rz^q,\calE_S)$. 

\subsection{Outlook: Generalized boundary problems in projected subspaces spaces}\label{sec:4.6}

In \cite{KSS} we introduced a calculus for constructing, in particular, parametrices for operators of the form 
 $$
     \begin{pmatrix}A\\ T\end{pmatrix}:
     \calW^{\mu,\mu}(\Omega)_S\lra
     \begin{matrix}L^2(\Omega)\\ \oplus\\ L^2(\partial\Omega,\cz^k)\end{matrix},
 $$ 
where $A$ is an edge-degenerate differential operator on a bounded domain (actually, a manifold with edges) and $T$ 
are so-called singular boundary conditions (for details we refer to \cite{KSS}). A limitation of this calculus is that 
it refers to spaces of the form $\calW^{\mu,\mu}(\Omega)_S$ and that (the invertibility of) the principal edge symbol 
also refers to $\calK^{\mu,\mu}(\rz_+)\oplus\calE_S$ with the same type $S$. As we have seen above, the natural 
domains of edge operators are not necessarily of this form, and the principal edge symbol can come along with a different 
asymptotic structure. Theorem \ref{thm:proj2} suggests to formulate a calculus for operators in projected subspaces, i.e., 
to consider operators of the form
 $$
     \begin{pmatrix}A\\ T\end{pmatrix}:
     P(\calW^{\mu,\mu}(\Omega)_S)\lra
     \begin{matrix}L^2(\Omega)\\ \oplus\\ L^2(\partial\Omega,\cz^k)\end{matrix},
 $$ 
where $P$ is a `singular projection' in $\calW^{\mu,\mu}(\Omega)_S$ associated with $A$, i.e., $P$ acts as the identity map on 
$\calW^{\mu,\mu}(\Omega)$. Note that this resembles the calculus introduced by Schulze in \cite{Schu3} for boundary value 
problems not requiring Shapiro-Lopatinskij ellipticity, where the classical boundary conditions are replaced by conditions in 
projected subspaces; however, in this set-up the projected spaces live over the boundary, while the spaces on $\Omega$ 
still are the classical Sobolev spaces. Besides usual interior ellipticity, the principal edge symbol should now be considered as a map 
 $$
     \begin{pmatrix}\sigma^\mu_\wedge(A)\\ \sigma^\mu_\wedge(T)\end{pmatrix}(y,\eta):
     \mathrm{range}\, \sigma^\mu_\wedge(P)(y,\eta)\lra 
     \begin{matrix}L^2(\rz_+)\\ \oplus\\ \cz^k\end{matrix},
$$
and ellipticity requires the invertibility of this map whenever $\eta\not=0$. Though the singular projection refers to spaces 
over $\Omega$, and not over $\partial\Omega$, note that the actual non trivial contribution only comes from the `singular 
part' of $P$ acting on $\calV^\mu(\Omega,\calE_S)$; the range of $\mathrm{range}\, \sigma^\mu_\wedge(P)(y,\eta)$ 
can be viewed as a subbundle of the trivial product bundle $S^*\partial\Omega\times\calE_S$ where $S^*\partial\Omega$ denotes 
the co-sphere bundle over the boundary. In this sense, the situation has some similarity with the one described for boundary 
conditions in projected spaces over the boundary. 


\bibliographystyle{amsalpha}

\end{document}